\documentclass{article}

\usepackage{amsmath,amsfonts,amsthm}
\usepackage{graphicx}
\usepackage{color}
\usepackage{enumerate}

\let\None\emptyset
\DeclareMathOperator{\DCP}{DCP}
\DeclareMathOperator{\CCLW}{CCLW}
\DeclareMathOperator{\CLQP}{CLQP}
\DeclareMathOperator{\rad}{rad}

\newcommand{\var}[1]{\mathcal{#1}}
\newcommand{\alphabet}[1]{\mathcal{#1}}
\newcommand{\alg}[1]{\mathbf{#1}}
\newcommand{\relstr}[1]{\mathbb{#1}}
\newcommand{\tuple}[1]{\mathbf{#1}}

\newcommand{\edge}[1]{\buildrel\relstr #1\over -}
\newcommand{\redge}[1]{\buildrel\relstr #1\over \rightarrow}
\newcommand{\ledge}[1]{\buildrel\relstr #1\over \leftarrow}
\newcommand{\edgev}{\mathrel{-}}
\newcommand{\redgev}{\mathrel{\rightarrow}}
\newcommand{\ledgev}{\mathrel{\leftarrow}}

\newcommand{\defeq}{\buildrel \rm def\over =}
\newcommand{\Q}{\mathbb Q}
\newcommand{\Z}{\mathbb Z}
\newcommand{\DC}{\mathbb D}

\theoremstyle{plain}
\newtheorem{theorem}{Theorem}[section]
\newtheorem{proposition}{Proposition}[section]
\newtheorem{example}{Example}[section]
\newtheorem{corollary}{Corollary}[section]
\newtheorem{question}{Question}[section]
\newtheorem{definition}[theorem]{Definition}
\newtheorem{claim}[theorem]{Claim}

\newtheorem{innercustomlemma}{Lemma}
\newenvironment{customlemma}[1]
  {\renewcommand\theinnercustomlemma{#1}\innercustomlemma}
  {\endinnercustomlemma}

\newtheorem{innerlemma}{Lemma}
\def\lemma[#1]#2\endlemma{%
  \begin{innerlemma}
    \label{#1}
    #2
  \end{innerlemma}
  \edef\tmp{\theinnerlemma}
  \expandafter\savelemma\expandafter{\tmp}{#1}{#2}
}
\def\savelemma#1#2#3{%
  \expandafter\def\csname lemma[#2]\endcsname{%
    \begin{customlemma}{#1}
      #3
    \end{customlemma}
  }
}

\def\remindlemma#1{\csname lemma[#1]\endcsname}

\title{Loop conditions with strongly connected graphs}
\author{Miroslav Ol\v s\'ak}

\begin{document}

\maketitle
\thanks{
  Partially supported
  by the Czech Grant Agency (GA\v CR) under grant no. 18-20123S,
  by the National Science Centre Poland under grant no. UMO-2014/13/B/ST6/01812
  and by the PRIMUS/SCI/12 project of the Charles University.
}
\begin{abstract}
  We prove that the existence of a term $s$ satisfying
  $s(r,a,r,e) = s(a,r,e,a)$ in a general algebraic structure is
  equivalent to an existence of a term $t$ satisfying
  $t(x,x,y,y,z,z)=t(y,z,z,x,x,y)$.
  As a consequence of a general version of this theorem and previous
  results we get that each strongly connected digraph of algebraic
  length one, which is compatible with an operation $t$ satisfying an
  identity of the from $t(\ldots)=t(\ldots)$, has a loop.
\end{abstract}

\section{Introduction}

Under which structural and algebraic conditions does a graph
compatible with an algebra necessarily have a loop? This question has
originated from the algebraic approach to the fixed--template
constraint satisfaction problems and answers have provided strong
results and useful tools in this area as well as in universal algebra,
see~\cite{DagstuhlPoly,DagstuhlAbs} for recent surveys.

A. Bulatov's algebraic refinement of the well--known Hell and Nesetril
dichotomy for undirected graphs~\cite{HellNesetril} can be stated as
follows.

\begin{theorem} [basic loop lemma]
  \cite{BulatovLoop}
  If a finite undirected graph $\relstr G$
  \begin{itemize}
  \item contains a triangle and
  \item is compatible with a Taylor term
    (that is, an operation satisfying a non-trivial set of identities
    of the form $t(\hbox{some vars})=t(\hbox{some vars})$),
  \end{itemize}
  then $\relstr G$ contains a loop (that is an edge joining a vertex with itself).
\end{theorem}

In 2010, M. Siggers observed that this structural result has an
algebraic equivalent statement.
\begin{theorem}
  \label{siggers-thm}
  \cite{Siggers}
  If $\alg A$ is a finite algebra with a Taylor term, then $\alg A$
  has a 6-ary Siggers
  term $s$ satisfying $s(x,x,y,y,z,z) = s(y,z,z,x,x,y)$.
\end{theorem}
\begin{proof}
  Let $\alg F$ be the $\alg A$-free algebra generated by the set
  $\{x,y,z\}$. Since $\alg A$ is finite, $\alg F$ is finite as well.
  We consider a symmetric binary relation (graph) on $\alg F$
  generated by the pairs
  $$
  \begin{pmatrix} x\cr y \end{pmatrix},
  \begin{pmatrix} x\cr z \end{pmatrix},
  \begin{pmatrix} y\cr z \end{pmatrix},
  \begin{pmatrix} y\cr x \end{pmatrix},
  \begin{pmatrix} z\cr x \end{pmatrix},
  \begin{pmatrix} z\cr y \end{pmatrix}.
  $$
  This graph has a triangle on $x,y,z$, so it has a loop $(l,l)$ by
  the loop lemma. Let $s$ be the six-ary term that generates the loop
  from the generators above. Then the following identity holds
  for the generators $x,y,z$.
  $$
  s(x,x,y,y,z,z) = l = s(y,z,z,x,x,y)
  $$
  However, since $\alg F$ is a free algebra, the identity holds in
  general.
\end{proof}

Notice that obtaining the structural result from the algebraic one is
even easier -- if we have a Siggers term $s$, and an undirected
triangle in a compatible graph, we find the loop directly by applying
the Siggers term to the six edges of the triangle (we use both
directions of each of three undirected edge).

The loop lemma was improved by L. Barto and M. Kozik.
\begin{theorem} [loop lemma]
  \label{dir-loop-lemma}
  \cite{BartoKozikLoop}
  If a finite digraph $\relstr G$
  \begin{itemize}
  \item is weakly connected,
  \item is smooth (has no sources and no sinks),
  \item has algebraic length 1 (cannot be mapped to a non-trivial
    directed cycle) and
  \item is compatible with a Taylor term,
  \end{itemize}
  then $\relstr G$ contains a loop.
\end{theorem}

Consequently, one can improve Theorem~\ref{dir-loop-lemma} as follows.
\begin{theorem}
  \cite{OptimalStrong}
  If $\alg A$ is a finite algebra with a Taylor term, then $\alg A$
  has a 4-ary Siggers term $s$ satisfying
  $s(r,a,r,e) = s(a,r,e,a)$.
\end{theorem}

Since both algebraic results give a Taylor term,
the following properties are equivalent for a finite algebra $\alg A$:
\begin{enumerate}[(i)]
\item $\alg A$ has a Taylor term.
\item $\alg A$ has a 6-ary Siggers term $s(x,x,y,y,z,z) = s(y,z,z,x,x,y)$.
\item $\alg A$ has a 4-ary Siggers term $s(r,a,r,e) = s(a,r,e,a)$.
\end{enumerate}

A natural question is whether the equivalence can be generalized to
general (infinite) algebras. A. Kazda found an example~\cite{Kazda} of
an algebra that has an idempotent Taylor term (that is, a Taylor term
t which additionally satisfies t(x,....x)=x), but no 6-ary or 4-ary
Siggers term. Therefore the remaining question is whether the
properties (ii) and (iii) are equivalent.

Recently, the author studied such conditions in
general~\cite{LoopConditions}. A \emph{loop condition} is
a condition of the form: There is a term $t$ satisfying
$t(\text{some variables}) = t(\text{some variables})$.
A consequence of results in~\cite{LoopConditions} is that the
following statements are equivalent for every algebra $\alg A$.
\begin{enumerate}[(i)]
\item $\alg A$ satisfies a non-trivial loop condition,
\item $\alg A$ has a 6-ary Siggers term $s(x,x,y,y,z,z) = s(y,z,z,x,x,y)$.
\item Every non-bipartite undirected graph compatible with an algebra
  $\alg B$ in the variety generated by $\alg A$ has a loop.
\end{enumerate}

In this paper, we give a positive answer to the raised question by
finding a broader class of loop conditions that are equivalent. This
allows us to give an even stronger structural characterization of
algebras satisfying some loop condition in the following form.
\begin{theorem}
  Let $\var V$ be a variety. Then the following are equivalent.
  \begin{enumerate}[(i)]
  \item $\var V$ satisfies a non-trivial loop condition,
  \item Let $\relstr G$ be a graph compatible with $\alg A\in\var V$.
    If $\relstr G$ has a strongly connected component with
    algebraic length one, then $\relstr G$ has a loop.
  \end{enumerate}
\end{theorem}

\subsection{Outline}

We start by providing a proper definition of the key concepts in
Section~\ref{preliminaries} including a digraph associated to a loop
condition. We also summarize there the results of the paper~\cite{LoopConditions}.

In Section~\ref{main-proof} we prove our main result: that all loop
conditions with a strongly connected digraph with algebraic length
one are equivalent.

Then we give two examples of classical classes that satisfy such a
condition in Section~\ref{example-classes}.

In Section~\ref{cyclic-terms}, we classify the strength of loop
comditions with strongly connected digraphs (of arbitrary algebraic
lengths)

Finally, in Section~\ref{no-strong-conn} we discuss the case
whithout strong connectedness. This case seems to be much colorful, so
we provide at least some counterexamples and some partial results.

\section{Preliminaries}
\label{preliminaries}

We refer to~\cite{Bergman,Bergman2} for undefined notions and more background.

\subsection{Digraphs, algebraic length}

A \emph{digraph} $\relstr G = (A,G)$ is a relational structure, where
$A$ is the set of nodes (vertices) and $G\subset A\times A$ is a binary relation,
in other words, the set of edges. We sometimes denote an edge
$(x,y)\in G$ of a digraph by
$$
x\redge G y,\text{ or }y\ledge G x.$$
If both
$(x,y),(y,x)\in G$, we write $x\edge G y$. If the digraph is clear from
the context, we may write just $x\redgev y, y\ledgev x$, or
$x\edgev y$. A loop in $\relstr G$ is an edge
$a\redgev a$.

If the relation $G$ is symmetric, we also call $\relstr G$ as
undirected graph.

Consider two digraphs $\relstr G=(A,G)$ and $\relstr H=(B,H)$. A digraph
homomorphism $f\colon \relstr G\to\relstr H$ is a mapping
$A\to B$ such that $f(a_1)\redge H f(a_2)$ whenever
$a_1\redge G a_2$.

Thorough the paper, we use the following basic digraphs:
\begin{enumerate}
\item A clique $\relstr K_n = (\{0,\ldots,n-1\}, \neq)$,
\item a symmetric cycle $\relstr C_n$ on $\{0,\ldots,n-1\}$, where
  $x\edgev y$ if $y\equiv x\pm1\pmod n$,
\item a directed cycle $\DC_n$ on $\{0,\ldots,n-1\}$, where
  $x\redgev y$ if $y\equiv x+1\pmod n$,
\item a directed path $\Z$ on integers
  $\{\ldots,-2,-1,0,1,2,\ldots\}$, where $k\redgev k+1$ for all
  integers $k$.
\end{enumerate}

An \emph{oriented walk} of length $n$
is a sequence of vertices $x_0,x_1,\ldots x_n$ such that for each $i$
we have $x_i\redgev x_{i+1}$ or $x_i\ledgev x_{i+1}$. A walk is
called \emph{directed} if for all $i$ we have the forward edge
$x_i\redgev x_{i+1}$. If $x_0=x_n$, we also talk about directed or
oriented cycle. Notice that a directed cycle as a walk of length $n$
on a digraph $\relstr G$ corresponds to a digraph homomorphism
$\DC_n\to\relstr G$.

If every pair of vertices is connected by an oriented walk, we call
the digraph \emph{weakly connected}. If every pair of vertices is
connected by a directed walk (in both ways), we call the digraph
\emph{strongly connected}.

\emph{Algebraic length} of an oriented walk is defined as the number
of forward edges minus the number of backward edges in the walk. The
\emph{algebraic length of a digraph} $\relstr G$, denoted
$al(\relstr G)$, is the greatest common divisor of the algebraic
lengths of all oriented cycles, or $\infty$ if all oriented cycles
have algebraic length zero. Note that our definition slightly differs
from the one from~\cite{BartoKozikLoop} on digraphs that are not weakly
connected. On the other hand, neither of papers is interested in that
case. If $\relstr G$ is weakly connected, then there is an oriented
$\relstr G$-cycle of length $al(\relstr G)$.

There are also nice characterizations of these properties using
digraph homomorphism. The algebraic length of a digraph $\relstr G$ is
the biggest $n$ such that there is a digraph homomorphism
$\relstr G\to\DC_n$, or $\infty$ if there is a digraph homomorphism
$\relstr G\to\Z$. Conversely, a finite digraph $\relstr G$ is
strongly connected if and only if there is a digraph homomorphism
$\DC_n\to\relstr G$ for some $n$ that is surjective on edges. We
finish this subsection with an even stronger property of finite
strongly connected digraphs with a given algebraic length.

\begin{proposition}
  Let $\relstr G$ be a finite strongly connected digraph and let 
  $S$ be the set of all numbers $n$ such that there is a digraph
  homomorphism $\DC_n\to\relstr G$ that is surjective on edges.
  Then $S$ contains only the multiples of $al(\relstr G)$ but it
  contains any such a multiple that is large enough.
\end{proposition}

\begin{proof}
Let $d$ denote the greatest common divisor of elements of $S$.
The set $S$ is closed under addition: Any two directed cycles
that covers all edges can be joined at any node. Therefore, by Schur's
theorem, the set $S$ contains all the multiples of $d$ that are large
enough. Every number of $S$ expresses an algebraic length of a
directed cycle, so $al(\relstr G)\mid d$. It remains to prove that
$d\mid al(\relstr G)$.

To see that, take $k\in S$, the corresponding directed cycle $c_k$
and an oriented cycle $c_{al}$ on $\relstr G$ of algebraic length
$al(\relstr G)$. In $c_{al}$, we append one copy of $c_k$ and replace
every backward edge by $(k-1)$ forward edges taken from $c_k$. The
result is a directed cycle $c$ that is surjective on edges and its
length is equal to $al(\relstr G)+nk$ for some $n$. Hence
$$
al(\relstr G)+nk\in S.
$$
Since $d\mid k$ and $d\mid al(\relstr G)+nk$, also
$d\mid al(\relstr G)$ and we are done.
\end{proof}

\subsection{Compatible digraphs, pp-constructions}

An $n$-ary operation $f\colon A^n\to A$ is said to be
\emph{compatible} with
an $m$-ary relation $R\subset A^m$ if
$f(\tuple r_1,\tuple r_2,\ldots,\tuple r_n)\in R$ for any
$\tuple r_1, \tuple r_2,\ldots,\tuple r_n\in R$. Here (and later as well) we abuse
notation and use $f$ also for the $n$-ary operation on $A^m$ defined
coordinate-wise.

An algebra $\alg A=(A,f_1,f_2,\ldots)$ is said to be compatible
with a relational structure $\relstr A=(A,R_1,R_2,\ldots)$ if all the
operations $f_1,f_2,\ldots$ are compatible with all the relations
$R_1,R_2,\ldots$.

We will extensively use a standard method for building compatible
relations from existing ones -- primitive positive (\emph{pp}, for
short) definitions.
A relation $R$ is \emph{pp-definable} from relations
$R_1,\ldots,R_n$ if it can be defined by a first order formula using
variables, existential quantifiers, conjunctions, the equality
relations, and predicates $R_1,\ldots,R_n$.
Clauses in pp-definitions are also referred to as \emph{constraints}.
Recall that if $R_1,\ldots,R_n$ are compatible with an algebra, then
so is $R$.

A pp-definition of a $k$-ary relation $R$ from a digraph $\relstr G$
can be described by a finite digraph $\relstr H$ with $k$
distinguished vertices $v_1,\ldots,v_k$. We define $R(x_1,\ldots,x_k)$
by the existence of a digraph homomorphism
$\relstr H\to\relstr G$ that maps $v_i$ to $x_i$. The edges of
$\relstr H$ correspond to the constraints in the
pp-definition, and the images of remaining vertices of $\relstr H$
(other that $v_1,\ldots,v_k$) are existentialy quantified.

Using a more general technique, we can also construct a relation that
is compatible with a subalgebra of a power $\alg A^k$. We pp-define
two relations from some relations compatible with $\alg A$: a $k$-ary
relation $B$, and a $nk$-ary relation $R$. Then $B$ is compatible with
$\alg A$, so it forms a subuniverse of $\alg A^k$. The $nk$-ary
relation $R$ is perceived as an $n$-ary relation on elements of
$\alg A^k$, so $R$ acts as a $n$-ary relation on $B$ as well. Even in
this perception, the relation $R$ is still compatible with the
subalgebra $\alg B\leq\alg A^k$ on subuniverse $B$. This construction
is a special case of pp-construction.

In our proof, we will use this construction in the following form.
Consider two finite (template) digraphs $\relstr V = (V_v, V_e)$,
$\relstr E = (E_v, E_e)$,
and two digraph homomorphisms
$\phi_0,\phi_1\colon \relstr V\to\relstr E$.
Then we take a digraph $\relstr G = (A,G)$ compatible with an algebra
$\alg A$, and construct a digraph $\relstr H = (B, H)$ where $B$ is the
set of all the mappings $\relstr V\to\relstr G$ and edges defined as
follows: The is an edge $v_0\redge H v_1$
if there is a homomorphism $e\colon \relstr E\to\relstr G$ such that
$v_0 = e\circ \phi_0$ and $v_1 = e\circ \phi_1$.
Then $B$ is a subuniverse of $\alg A^{|V_v|}$, and $H$ a compatible
digraph with the appropriate algebra $\alg B$.

\subsection{Loop conditions}

A \emph{variety} $\var V$ is a class of algebras closed under powers,
subalgebras and homomorphic images. In every variety, for any set $S$,
we can find the \emph{free algebra $\alg F$ freely generated by $S$}
with the following property: Let $t_1.t_2$ be any $n$-ary term
operations, and $s_1,\ldots,s_n$ be distinct elements of $S$. If
$t_1(s_1\ldots,s_n) = t_2(s_1,\ldots,s_n)$ in $\alg F$, then
$t_1(x_1\ldots,x_n) = t_2(x_1,\ldots,x_n)$ for any $x_1,\ldots,x_n$ in
any algebra from $\var V$.

A \emph{loop condition} is given by a set of variables $V$ and two
$n$-tuples
$$
x_1,x_2,\ldots,x_n,y_1,y_2,\ldots,y_n\in V.
$$
An algebra
(or variety) is said to satisfy such a condition if there is a term
$t$ in the algebra (variety) satisfying the identity
$$
t(x_1,x_2,\ldots,x_n) = t(y_1,y_2,\ldots,y_n).
$$
To such a loop condition $C$, we assign a digraph
$$
\relstr G_C = (V,\{(x_1,y_1),(x_2,y_2),\ldots,(x_n,y_n)\}
$$
\def\siggerspic #1#2#3#4#5#6{
\begingroup
\setlength{\unitlength}{57pt}
\begin{picture}(0.9,0.8)%
    \put(0.52,0.67){\makebox(0,0)[b]{\smash{$#3$}}}%
    \put(0.16,0.04){\makebox(0,0)[b]{\smash{$#1$}}}%
    \put(0.85,0.04){\makebox(0,0)[b]{\smash{$#5$}}}%
    \put(0.34,0.38){\rotatebox{60}{\makebox(0,0)[b]{\smash{$#2$}}}}%
    \put(0.67,0.39){\rotatebox{-60}{\makebox(0,0)[b]{\smash{$#4$}}}}%
    \put(0.51,0.03){\makebox(0,0)[b]{\smash{$#6$}}}%
\end{picture}%
\endgroup}

\def\siggerseq#1{%
\vbox{%
  \def\V{{\downarrow}}%
  \halign{\hfill$##{}$&&\hfil${}##{}$\hfil\cr
    #1%
}}}

\begin{figure}[!ht]
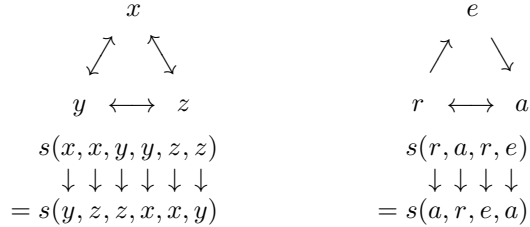

\centerline{\vbox{%
\halign{\hfil#\hfil&\kern2cm\hfil#\hfil\cr
\siggerspic y\longleftrightarrow x\longleftrightarrow z
            \longleftrightarrow
&%
\siggerspic r\longrightarrow e\longrightarrow a
    \longleftrightarrow
\cr
\noalign{\medskip}
\siggerseq{
  & s(&x&,&x&,&y&,&y&,&z&,&z&)\cr
  &  &\V&&\V&&\V&&\V&&\V&&\V&  \cr
 =& s(&y&,&z&,&z&,&x&,&x&,&y&)\cr
}
&%
\siggerseq{
  & s(&r&,&a&,&r&,&e&) \cr
  &  &\V&&\V&&\V&&\V&  \cr
 =& s(&a&,&r&,&e&,&a&) \cr
}
\cr}}}
\caption{Graphs corresponding to the 6-ary and 4-ary Siggers terms}
\end{figure}

Since $\alg A$ satisfies a certain loop condition if and only if the
variety generated by $\alg A$ satisfies it, it suffices to focus on
varieties.

\begin{proposition}
  \label{lc-basic-char}
The following are equivalent for a variety $\relstr V$ and a loop
condition $C$.
\begin{enumerate}
\item $\var V$ satisfies $C$.
\item Let $\alg F$ be the free algebra freely generated by the set of
  variables in $C$, and let $\relstr G$ be a digraph compatible with $\alg F$
  containing all the edges of $\relstr G_c$. Then $\relstr G$ have a loop.
\item For any algebra $\alg A\in\var V$ and any digraph $\relstr G$
  compatible with $\alg A$, if there is a digraph homomorphism
  $\relstr G_C\to\relstr G$, then $\relstr G$ has a loop.
\end{enumerate}
\end{proposition}

The proof of the proposition is just a direct generalization of the
proof of Theorem~\ref{siggers-thm}. For a general proof, we refer the
reader to~\cite{LoopConditions}.

It is apparent from the proposition that a validity of a loop
condition $C$ in a variety (algebra) is determined by the
digraph $\relstr G_C$. Based on that tight connection between a loop
condition and its digraph, we sometimes assign digraph attributes
(strongly connected, algebraic length) to a loop condition, meaning
the attributes of the corresponding digraph.
For a general finite digraph $\relstr G$ we also talk about the
$\relstr G$ loop condition -- that refers to any loop condition $C$
such that $\relstr G_C \simeq \relstr G$.

A loop condition is called \emph{trivial} if it is satisfied in every
algebra, equivalently if its digraph contains a loop. Otherwise, the
loop condition is called \emph{non-trivial}.

A simple reason for implication between loop conditions,
directly obtained from item (iii) of Proposition~\ref{lc-basic-char},
is the following theorem.
\begin{proposition}
  \label{homo-implication}
  If there is a digraph homomorphism $\relstr G\to\relstr H$, then
  every variety (algebra) satisfying the $\relstr G$ loop condition satisfies the
  $\relstr H$ loop condition as well.
\end{proposition}

Nevertheless, this is not by far the only way for obtaining an
implication between loop conditions. It is proved
in~\cite{LoopConditions} that all non-trivial undirected non-bipartite
loop conditions are equivalent. If a variety (or an algebra)
satisfies a non-trivial loop condition, we call it
\emph{loop-producing}. Since any finite loopless digraph can be
homomorphically mapped to a large enough clique, we get the following
corollary of that result.

\begin{proposition}
  \label{loop-prod-char}
  Let $\var V$ be a variety. The following are equivalent.
  \begin{enumerate}
  \item $\relstr V$ is a loop-producing variety,
  \item Let $\relstr K_\omega$ denote the infinite countable clique,
    and let $\alg F$ be a $\var V$-free algebra freely generated by
    the nodes of $\relstr K_\omega$. Then the graph generated by edges
    of $\relstr K_\omega$ has a loop.
  \item $\relstr V$ has a 6-ary Siggers term $s(x,x,y,y,z,z)=s(y,z,z,x,x,y)$,
  \item $\relstr V$ satisfies all undirected non-bipartite loop
    conditions,
  \item Whenever $\alg A\in\var V$ is compatible with a digraph
    $\relstr G$ such that there is homomorphism
    $\relstr C_{2n+1}\to \relstr G$ for some $n$, then $\relstr G$ has
    a loop.
  \end{enumerate}
\end{proposition}

Less formally, it suffices to test the loop-producing property on
large cliques, but then it can be applied to any symmetric cycle of
odd length.

\section{The proof of the loop condition collapse}
\label{main-proof}

In this section, we prove the following.

\begin{theorem}
\label{main-thm}
Let $\relstr G$ be a strongly connected digraph with algebraic
length 1, and $\var V$ be a loop-producing variety. Then $\var V$
satisfies the $\relstr G$ loop condition.
\end{theorem}

From basic properties of loop conditions we then obtain the following
corollaries.

\begin{corollary}
Let $\alg A$ be a loop-producing algebra compatible with a digraph
$\relstr G$. If $\relstr G$ has a strongly connected component with
algebraic length one, then $\relstr G$ has a loop.
\end{corollary}
\begin{proof}
  It follows from Theorem~\ref{main-thm} and
  Proposition~\ref{lc-basic-char}, implication $(1)\Rightarrow(3)$.
\end{proof}
\begin{corollary}
  \label{loop-cond-eq}
All the non-trivial loop conditions that have a strongly connected
component of algebraic length one are equivalent.
\end{corollary}
\begin{proof}
  Any variety satisfying such a loop conditions is
  loop-producing by non-triviality. Conversely, any loop-producing
  variety satisfy all the loop conditions with algebraic length one by
  Theorem~\ref{main-thm}. Finally, all the non-trivial loop conditions
  that have a strongly connected component of algebraic length one are
  satisfied by Proposition~\ref{homo-implication}.
\end{proof}
\begin{corollary}
  Any loop-producing variety has a 4-ary Siggers term
  $s(r,a,r,e) = s(a,r,e,a)$.
\end{corollary}
\begin{proof}
  This follows directly from Theorem~\ref{main-thm} since the
  rare-area loop condition is an example of a strongly connected loop
  condition with algebraic length one.
\end{proof}

\subsection{Digraph definitions}

We provide three concrete types of digraphs that will serve as
intermediate steps in the proof of Theorem~\ref{main-thm}.

\begin{definition}
Given cycle lengths $a,b\geq 1$, we define the digraph $\DCP(a,b)$
(directed cycle pair) as follows. The nodes are
$A_0, \ldots, A_{a-1}$ and $B_0, \ldots, B_{b-1}$ such that
$A_0 = B_0$ while all the other are different. Edges goes from $A_i$
to $A_{i+1}$ modulo $a$ and from $B_i$ to $B_{i+1}$ modulo $b$.
\end{definition}

\medskip
\begin{figure}[!ht]
  \def\svgwidth{5cm}
  \centering
     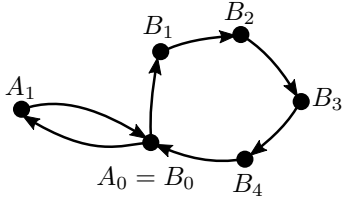
  \caption{The digraph $\DCP(2,5)$.}
\end{figure}

\begin{definition}
For a given path length $k\geq1$, number $l\geq0$ of loop symbols, and
cycle size $c\geq1$, the digraph $\CCLW(k, l, s)$ (cycle walks) is
defined as follows.
Let $\alphabet A = \{0,1,\ldots,c-1\}$ and $\alphabet L$ be an alphabet
of size $l$ not containing $\None$. The nodes of $\CCLW(k, l, c)$ are
the sequences $(a_1, l_1, a_2, l_2, \ldots. l_{k-1}, a_k)$ such that:
\begin{itemize}
\item every $a_i\in\alphabet A$ and every $l_i\in\alphabet L\cup\{\None\}$,
\item loop symbols do not repeat,
      that is, if $l_i=l_j$, then $i=j$ or $l_i=l_j=\None$,
\item If $l_i\in\alphabet L$, then $a_i=a_{i+1}$. Otherwise
$a_i\equiv a_{i+1}\pm1\pmod c$.
\end{itemize}
There is an edge $(n_1, n_2)\in \CCLW(k, l, s)$ if there is a node $n$
in $\CCLW(k+1, l, s)$ such that $n_1$ is a prefix of $n$ and $n_2$ is
a suffix of $n$.
\end{definition}

\begingroup
\everymath={\scriptstyle}
\medskip
\begin{figure}[!ht]
  \def\svgwidth{12cm}
  \centering
     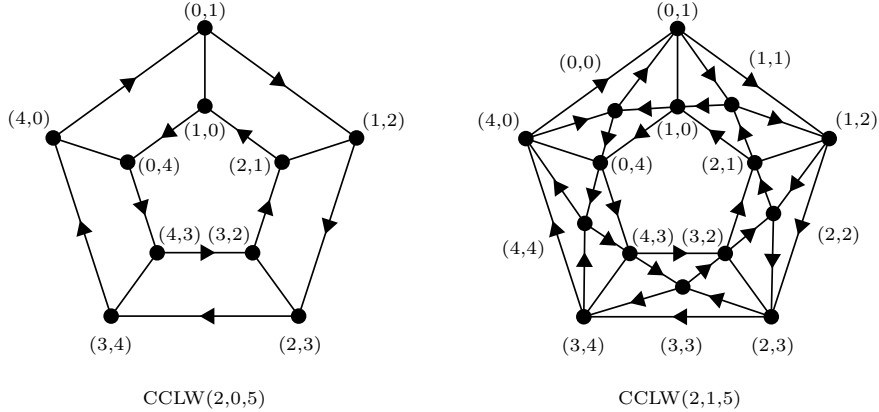
  \caption{Examples of digraphs of cycle walks. The loop
  symbol is ommited since it is uniquely determined by adjacent
  items.}
\end{figure}

\endgroup

\begin{definition}
For a given path length $k\geq1$, number $l\geq0$ of loop symbols, and
alphabet size $s\geq1$, the digraph $\CLQP(k, l, s)$ (clique paths) is
defined as follows.
Let $\alphabet A$ be an alphabet of size $s$ and $\alphabet L$ be an alphabet
of size $l$ not containing $\None$. The nodes of $\CLQP(k, l, s)$ are
the sequences $(a_1, l_1, a_2, l_2, \ldots. l_{k-1}, a_k)$ such that:
\begin{itemize}
\item every $a_i\in\alphabet A$ and every $l_i\in\alphabet L\cup\{\None\}$,
\item loop symbols do not repeat,
      that is, if $l_i=l_j$, then $i=j$ or $l_i=l_j=\None$,
\item letters repeat if and only if they are connected by loop
  symbols, that is, for any $i<j$: $a_i=a_j$ if and only if all
$l_i,\ldots,l_{j-1}\in\alphabet L$.
\end{itemize}
There is an edge $(n_1, n_2)\in \CLQP(k, l, s)$ if there is a node $n$
in $\CLQP(k+1, l, s)$ such that $n_1$ is a prefix of $n$ and $n_2$ is
a suffix of $n$.
\end{definition}

Note that cycle walks and clique paths differ just in the third point
where cycle walks allows repetition of letters but are more
restrictive on the local behavior.

\subsection{Proof outline}

Let $\var V$ be a fixed loop-producing variety.
We begin by stating seven lemmas.

\lemma[lm-start]
  The $\CLQP(1,0,3)$ loop condition is satisfied in $\var V$.
\endlemma
\lemma[lm-clqp-raise]
  For any $k\geq1, s\geq1$, the following implication between loop
  conditions holds for $\var V$:
  $$\CLQP(k,0,s)\Rightarrow\CLQP(k+1,s,1).$$
\endlemma
\lemma[lm-clqp-reduce]
  For any $k\geq2, l\geq1, s\geq1$, the following implication between loop
  conditions holds for $\var V$:
  $$\CLQP(k,l,s)\Rightarrow\CLQP(k, l-1,3s).$$
\endlemma
\lemma[lm-clqp-cclw]
  For any $k\geq1, l\geq0, c\geq1$, the following implication between loop
  conditions holds for $\var V$:
  $$\CLQP(k,l,1)\Rightarrow\CCLW(k,l,c).$$
\endlemma
\lemma[lm-cclw-reduce]
  For any $k\geq2, l\geq1, c\geq3$, $c$ odd, the following implication
  between loop conditions holds for $\var V$:
  $$\CCLW(k,l,c)\Rightarrow\CCLW(k,l-1,c).$$
\endlemma
\lemma[lm-cclw-dcp]
  For any odd $c\geq3$ there is some $k\geq2$ such that the following implication
  between loop conditions holds for $\var V$:
  $$\CCLW(k,0,c)\Rightarrow\DCP(2,c).$$
\endlemma
\lemma[lm-dcp-raise]
  For any $p,b\geq1$, $b$ odd, the following implication
  between loop conditions holds for $\var V$:
  $$\DCP(2^p,2^{p+1}+b)\Rightarrow\DCP(2^{p+1},b).$$
\endlemma

First, let us see how Theorem~\ref{main-thm} follows from these lemmas.
Using Lemma~\ref{lm-clqp-raise} once and Lemma~\ref{lm-clqp-reduce}
repeatedly yields the implication
$\CLQP(k,0,s)\Rightarrow\CLQP(k+1,0,3^s)$
for any $k,s\geq 1$.
Since the $\CLQP(1,0,3)$ loop condition holds in $\var V$ by
Lemma~\ref{lm-start}, for any $k\geq1$ there is $s\geq3$ such that
the $\CLQP(k,0,s)$ loop condition also holds in $\var V$.

Therefore by Lemmas~\ref{lm-clqp-raise},~\ref{lm-clqp-cclw}, also
the $\CCLW(k,l,c)$ loop condition holds in $\var V$ for any
$k\geq2, c\geq1$ and some $l$ depending on $k$. If $c$ is
odd and at least 3, we can reduce the number of loop symbols $s$ to zero by
repeated application of Lemma~\ref{lm-cclw-reduce}. So
the $\CLQP(k,0,c)$ loop condition holds in $\var V$ for any $k\geq2, c\geq3$, $c$ odd.
Thus by Lemma~\ref{lm-cclw-dcp} also the $DCP(2,c)$ loop condition holds for any odd
$c\geq3$.

Now, the $\DCP(2^p, c)$ loop condition holds for any odd $c\geq3$ and for any
$p\geq1$. This is done by induction on $p$ where the induction step is
given by Lemma~\ref{lm-dcp-raise}.
Finally, consider any strongly connected loopless digraph $\relstr G$
with algebraic length 1. Digraph $\relstr G$ is a homomorphic image of
a directed $2^p$-cycle for large enough $p$. Also $\relstr G$ is a
homomorphic image of a directed $c$-cycle where $c$ is odd and large
enough.  Therefore $\relstr G$ is a homomorphic image of
$\DCP(2^p, c)$ for large enough $p,c$, so the $\relstr G$ loop
condition is satisfied in $\var V$.

It remains to check the validity of our 7 lemmas.

Lemma~\ref{lm-start} is satisfied by Proposition~\ref{loop-prod-char}
since $\CLQP(1,0,3)$ is isomorphic to the undirected
triangle. Lemma~\ref{lm-clqp-raise} also
follows from an isomorphism of the appropriate digraphs. Sequence 
$(a_1, \None, a_2, \None, \ldots, a_k)$ in $\CLQP(k,0,s)$ can be
identified with the sequence $(x, a_1, x, a_2, \ldots, a_k, x)$ in
$\CLQP(k+1, s, 1)$ where $x$ is the only letter in the apropriate
alphabet.

Lemma~\ref{lm-clqp-cclw} is simple as well. the $\CCLW(k,l,c)$ loop
condition is implied by $\CLQP(k,l,1)$ just because 
$\CLQP(k,l,1)$ is a subdigraph of $\CCLW(k,l,c)$.

Lemmas~\ref{lm-clqp-reduce}, \ref{lm-cclw-reduce}, \ref{lm-cclw-dcp},
\ref{lm-dcp-raise} deserve bigger attention.

\subsection{Proof of Lemma~\ref{lm-clqp-reduce}}

\remindlemma{lm-clqp-reduce}

Assume that the $\CLQP(k,l,s)$ loop condition is satisfied in $\var V$. To
show that the $\CLQP(k,l-1,3s)$ loop condition is satisfied in $\var V$,
consider a digraph $\relstr G$ compatible with an algebra in $\var V$
such that $\CLQP(k, l-1, 3s)$ is a subdigraph of $\relstr G$. It
suffices to show that $\relstr G$ has to contain a loop.

Let us construct another digraph $\relstr H$, pp-constructed from
$\relstr G$. Nodes of $\relstr H$ are defined as the homomorphisms
$\CLQP(k,l-1,s)\to \relstr G$.

To define the edges of $\relstr H$, we investigate the digraph
$\CLQP(k,l,s)$ in more detail. Let $L_1, \ldots, L_l$ be the
loop symbols of $\CLQP(k,l,s)$ and $\{L_1, \ldots, L_{l-1}\}$ be
the loop symbols of $\CLQP(k,l-1,s)$. Then we can view
$\CLQP(k,l-1,s)$ as an induced subgraph of $\CLQP(k,l,s)$. There are no
extra edges since $k\geq2$. Let $\relstr V$ denote the digraph induced
on all nodes of $\CLQP(k,l,s)$ containing the symbol $L_l$. So the set
of all vertices is decomposed into $\CLQP(k,l-1,s)$ and $\relstr V$.

Consider $v_0,v_1\colon \CLQP(k,l-1,s)\to \relstr G$, two vertices of
$\relstr H$. By definition, there is an edge $v_0\redge H v_1$ if and
only if there is a digraph homomorphism $f\colon\relstr V\to\relstr G$
such that for any edge $(x,y)$ in $\CLQP(k,l,s)$:
\begin{itemize}
\item If $x\in\CLQP(k,l-1,s)$ and $y\in\relstr V$, then $v_0(x)\redge G f(y)$.
\item If $y\in\CLQP(k,l-1,s)$ and $x\in\relstr V$, then $v_1(y)\ledge G f(x)$.
\end{itemize}

\begin{claim}
  There is an undirected triangle in $\relstr H$.
\end{claim}

To show the claim, let $\alphabet A_s$, $\alphabet A_{3s}$ be alphabets of
$\CLQP(k,l-1,s)$, $\CLQP(k,l-1,3s)$ respectively. Fix three injective
mappings $u_1,u_2,u_3\colon \alphabet A_s\to\alphabet A_{3s}$ with pairwise disjoint
images. These mappings extend naturally to three homomorphisms
$\CLQP(k,l-1,s)\to\CLQP(k,l-1,3s)$, which are vertices of $\relstr H$
since $\CLQP(k,l-1,3s)$ is a subdigraph of $\relstr G$. We claim that
these three nodes are pairwise adjacent. By symmetry, it suffices to
show that there is an edge $(u_1,u_2)$ in $\relstr H$.
We define a homomorphism $f\colon \relstr V\to\CLQP(k,l-1,3s)$ as
follows. A general node of $\relstr V$ of the form
$$
(a_1, l_1, \ldots, a_{i}, l_i=L_l, a_{i+1}, l_{i+1}, \ldots, a_n)
$$
is mapped to the node
$$
(u_1(a_1), l_1, \ldots, u_1(a_{i}), \None, u_2(a_{i+1}), l_{i+1}, \ldots, u_2(a_n))
$$
Such $f$ is a digraph homomorphism $\relstr V\to\relstr G$ satisfying
the required properties, so there is an edge $u_1\redge H u_2$.

Since the variety $\var V$ is a loop-producing variety and digraph
$\relstr H$ contains an undirected triangle,
$\relstr H$ contains a loop by Proposition~\ref{loop-prod-char}. By
definition of $\relstr H$, such a loop
witnesses the existence of a homomorphism $\CLQP(k,l,s)\to G$. So by
the $\CLQP(k,l,s)$ loop condition, there is a loop in
$\relstr G$.

\subsection{Proof of Lemma~\ref{lm-cclw-reduce}}

\remindlemma{lm-cclw-reduce}

The approach is very similar to the proof of
Lemma~\ref{lm-clqp-reduce}. However, we will repeat the reasoning because
of the slight differences.

Assume that the $\CCLW(k,l,c)$ loop condition is satisfied in
$\var V$ and $\CCLW(k,l-1,c)$ is a subdigraph of a digraph $\relstr G$
compatible with an algebra in $\var V$. We want to find a loop in
$\relstr G$. We pp-construct a digraph $\relstr H$. Its nodes are all
the digraph homomorphisms $\CCLW(k,l-1,c)\to\relstr G$.

The edges of $\relstr H$ are defined analogously to the previous
case. We decompose the nodes of $\CCLW(k,l,c)$ into
two induced subraphs: $\CCLW(k,l-1,c)$ and $\relstr V$. There is an
edge $A\redge H B$ if there is a digraph homomorphism
$f\colon\relstr V\to\relstr G$ such that the pair of homomorphisms
$(A, f)$ preserves the edges from $\CCLW(k,l-1,c)$ to $\relstr V$ and
the pair of homomorphisms $(B, f)$ preserves the edges in the other
direction.

\begin{claim}
  Digraph $\relstr H$ contains an undirected cycle of length $c$.
\end{claim}

Indeed, let $\alphabet A$ denote the set
$\{0, 1, \ldots, c-1\}$. Let $u_i\colon\alphabet A\to\alphabet A$
denote the mapping
$$
u_i(x) \equiv i+x \pmod c.
$$
Since these mappings are also endomorphisms of $\relstr C_n$,
these mappings can be extended to homomorphisms
$\CCLW(k,l-1,c)\to\CCLW(k,l-1,c)$. So we can view them as nodes
of $\relstr H$.

There are edges $u_x\redge H u_y$ for
$y=x\pm1\pmod c$. To verify it, we construct a
homomorphism $\relstr V\to \CCLW(k,l-1,c)$ testifying that. An element
of $\relstr V$
$$
(a_1, l_1, \ldots, a_{i}, l_i=L_l, a_{i+1}, l_{i+1}, \ldots, a_n),
$$
where $L_l$ denotes the one extra loop symbol, is mapped to
$$
(u_x(a_1), l_1, \ldots, u_x(a_{i}), l_i=\None, u_y(a_{i+1}), l_{i+1}, \ldots, u_y(a_n))
$$
in $\CCLW(k,l-1,c)$. Note that the $a_i=a_{i+1}$ because of the loop
symbol $L_i$, so $u_y(a_{i+1})\equiv u_x(a_{i})\pm 1$. Therefore the
position $i$ meets the requirement for $\CCLW(k,l-1,c)$.

Since there is a symmetric cycle $\relstr C_c$ in $\relstr H$ and
$\var V$ is a loop-producing variety, there is a loop in $\relstr H$
by Proposition~\ref{loop-prod-char}~(5). A loop in $\relstr H$
corresponds to a digraph homomorphism
$\CCLW(k,l,c)\to\relstr G$. Finally, the $\CCLW(k,l,c)$ loop condition
gives a loop in $\relstr G$.

\subsection{Proof of Lemma~\ref{lm-cclw-dcp}}

\remindlemma{lm-cclw-dcp}

Actually, this is true because there is a digraph homomorphism
$\CCLW(k,0,c)\to\DCP(2,c)$. To show it, let us reformulate the
existence of a homomorphism a bit.

Consider a digraph homomorphism
$a\colon\relstr Z\to\relstr C_c$ which contains
all the finite walks on $\relstr C_c$. This
is possible since there are just countable number of the finite walks
and $\relstr C_c$ is connected. Let us view the mapping $a\colon\relstr
Z\to\relstr C_c$ (and later defined $x\colon\relstr Z\to\DCP_c$ as
well) as a sequence instead of a function and write $a_i$ instead of
$a(i)$.

\begin{claim}
  There is a digraph homomorphism $x\colon\relstr Z\to\DCP(2,c)$ and a
  number $k\geq 1$ such that whenever
  $$
  (a_{i-k}, a_{i-k+1}, \ldots, a_{i+k}) = (a_{j-k}), a_{j-k+1}, \ldots, a_{j+k})
  $$
  for some integers $i,j$, then also $x_i = x_j$.
\end{claim}

First, let us see how the claim proves Lemma~\ref{lm-cclw-dcp}. The
value $x_i$ is determined only by the
$(2k+1)$-tuple $(a_{i-k}, \ldots, a_{i+k})$. By construction of $a$,
any such walk on $\relstr C_c$ of length $(2k+1)$ occurs among such
tuples. So there is a well defined mapping
$\CCLW(2k+1,0,c)\to\DCP(2,c)$ given by
$$
(a_{i-k}, \None, a_{i-k+1}, \None, \ldots, a_{i+k}) \mapsto x_i.
$$
Finally, this mapping is a homomorphism since the sequence $a$ covers
all the walks of length $(2k+2)$ corresponding to the edges of
$\CCLW(2k+1,0,c)$.

Now, we prove the claim. Let $S$ be the set of all the integers $i$
satisfying
\begin{enumerate}
\item $a_i = 0$ or
\item $a_i\equiv0\pmod 2$ and all the values
  $a_{i-c+1}, \ldots, a_{i+c-1}$ are nonzero.
\end{enumerate}
Observe that the difference between consecutive elements of $S$ can
not exceed $2c$. Indeed, whenever there are two zeros $a_i$, $a_j$ in
the sequence $a$, $i+2c<j$, and there are no zeros between them, there
is at least one integer in the middle assigned to $S$ by the rule 2.

On the other hand, whenever there are two integers in $S$ such that
their difference is less than $c$, then the difference is even.

Given these facts, we can define the sequence $x$. For every $i\in S$
we set $x_i = A_0 = B_0$, then we fill the even spaces by alternation
$$
A_0, A_1, A_0, A_1, \ldots, A_0
$$
and the odd spaces greater than $c$ by
$$
B_0, B_1, \ldots B_{c-1}, A_0, A_1, A_0, A_1, \ldots, A_0.
$$
Such a definition of $x_i$ depends only on
$S\cap\{i-2c+1,\ldots,i+2c-1\}$ and the definition of $S$ depends just
on neighborhoods with radius $c$. So $k=3c$ completes the claim.

\subsection{Proof of Lemma~\ref{lm-dcp-raise}}

\remindlemma{lm-dcp-raise}

Let the $\DCP(2^p,2^{p+1}+b)$ loop condition be satisfied in $\var V$ and
let $\relstr G$ be a digraph compatible with an algebra
$\alg A\in\var V$ having $\DCP(2^{p+1},b)$ as its subdigraph. We need to
show that $\relstr G$ contains a loop. 

Every node in $\DCP(2^{p+1},b)$ is contained in a directed cycle
of length $2^{p+1}+b$. Since this property is pp-definable, we can
assume without lost of generality that all nodes of $\relstr G$ are
contained in such a cycle.

On the same set of vertices, we define another digraph $\relstr H$
using second relational power of the edges of $\relstr G$, that is
$(x,y)$ is an $\relstr H$-edge if and only if there is a directed
$\relstr G$-walk of length 2 from $x$ to $y$. Since $\relstr H$ is
pp-defined from $\relstr G$, it is compatible with $\alg A$.

It is still true that every node
of $\relstr H$ is contained in a directed $\relstr H$-cycle of length
$2^{p+1}+b$. Moreover, the nodes
$A_0, A_2, A_4, \ldots, A_{2^{p+1}-2}$ form a cycle of length
$2^p$. Therefore, there is a homomorphism
$\DCP(2^p,2^{p+1}+b)\to\relstr H$. Using the assumption on $\var V$,
we get a loop in $\relstr H$.

Such a loop corresponds to an undirected edge in $\relstr G$. Any
directed cycle of an even length can be homomorphically mapped to such
an edge, in particular the cycle of length $2^p$. Since every node of
$\relstr G$ belongs to a cycle of length $2^{p+1}+b$, there is a
digraph homomorphism $\DCP(2^p,2^{p+1}+b)\to\relstr G$. Second
application of the loop condition finished the proof.

\section{Example loop-producing classes}
\label{example-classes}

In this section, we compare our result with standard classes
in universal algebra, and show which of them are loop-producing.

We start by providing an alternative proof of
Theorem~\ref{siggers-thm} that finite Taylor algebras have a Siggers
term. Our approach has a slightly weaker outcome than
Theorem~\ref{dir-loop-lemma} (finite directed loop lemma) since we
require strong conectedness. On the other hand, we managed to reduce
the necessity of finiteness to a single step in the proof.
That gives a hope for generalizations under weaker assumptions than
local finiteness.
One such generalization could be into the class of so called
oligomorphic structures, which are already known to satisfy their own
version of loop lemma, so called pseudoloop lemma. Details can be
found in~\cite{PseudoLoop}.

A variety $\var V$ is said to be \emph{locally finite}, if any finitely
generated $\alg A\in\var V$ is finite. Conversely, any variety
generated by a single finite algebra is locally finite, which explains
the equivalence between Theorem~\ref{siggers-thm} and the following
formulation.

\begin{theorem}
Let $\var V$ be a locally finite variety having a (not necessarilly
idempotent) Taylor term, that is, a $k$-ary term $t$ satisfying some $k$
identities of the form
\begin{align*}
t(x,?,?,...,?,?) &= t(y,?,?,...,?,?) = s_1(x,y),\cr
t(?,x,?,...,?,?) &= t(?,y,?,...,?,?) = s_2(x,y),\cr
&\kern 1.7mm\vdots\cr
t(?,?,?,...,?,x) &= t(?,?,?,...,?,y) = s_k(x,y).\cr
\end{align*}
where every question mark stands for either $x$ or $y$. Then $\var V$
is loop-producing.
\end{theorem}
\begin{proof}
  We check that by verifying that $\var V$ satisfies the $\relstr
  K_{3k}$ loop condition. Let $\alg F$ be
  the free algebra generated by $3k$ elements, and $\relstr G$ be a
  graph compatible with it containing $3k$-clique as a homomorphic
  image. We are supposed to prove that $\relstr G$ contains a loop.

  To do that, it is sufficient to prove that whenever $\relstr G$
  contains a clique of size $c\geq 3k$ as a homomorphic image, it also
  contains a clique of size $c+1$. Therefore we will be able to
  gradually increase the size of the clique, and eventually exceed the
  size of algebra $\alg F$ which is finite since $\var V$ is a locally
  finite variety. The existence of a loop is then a direct
  consequence of the pigeonhole principle.

  So let us suppose that there is a clique with elements
  $A_i, B_i, C_i, D_j$ for $1 \leq i\leq k$ and $1\leq j\leq c-3k$.
  We are going to fill a matrix $a_{i,j}$ of size $c+1\times k$ in such
  a way that the Taylor term $t$ applied row-wise outputs the desired
  clique of size $c+1$.
  We write $a_{ii} = A_i, a_{ij} = B_i$ for $i\neq j$,
  $1\leq i,j\leq k$.
  So the values $t(a_{i,1},\ldots,a_{i,k})$ form a clique for
  $i\leq k$ since we used disjoint parts of the original clique
  in separate rows.
  Now, we fill in the next $k$ rows of the matrix.
  For a row $k+i$, where $1\leq i\leq k$, we use
  the elements $A_i,C_i$ in such a way that
  $$
  t(a_{k+i,1}, a_{k+i,2}, \ldots, a_{k+i,k}) = s_i(A_i, C_i)
  $$
  and moreover $a_{k+i,i} = A_i$. This is possible by the $i$-th
  Taylor identity. If we apply the Taylor term to the
  first $2k$ rows now, we still get a clique. In particular, there is
  an edge between
  $$
  t(a_{i,1}, a_{i,2}, \ldots, a_{i,k}) =
  t(B_i, \ldots, B_i, A_i=a_{i,i}, B_i, \ldots, B_i)
  $$
  and
  $$
  t(a_{k+i,1}, a_{k+i,2}, \ldots, a_{k+i,k}) = s_i(A_i, C_i)
  $$
  since there is a representation of
  $$
  s_i(A_i, C_i) = t(b_{i,1}, b_{i,2}, \ldots, b_{i,k})
  $$
  such that all $b_{i,j}$ are equal to $A_i$ or $C_i$ and moreover
  $b_{i,i}=C_i$.

  Since $a_{i,i}=a_{k+i,i} = A_i$, every column contains at most $2n-1$
  distinct vertices. Hence, we just complete the $c+1-2n$ remaining
  positions by the remaining $c-(2n-1)$ distinct vertices, and the Taylor term
  applied to rows gives a clique of size $c+1$. By this process, we
  gradually raise the size of the clique until we exceed the size of
  the algebra $\alg F$. Thus, we get a loop, and the variety is
  loop-producing.
\end{proof}

If we omit the assumption of local finiteness, we cannot hope for the
Siggers term even under the assumption of idempotency, as shown
in~\cite{Kazda}. This counterexample can be even extended to a
congruence meet-semidistributive variety that does not have a Siggers
term. That directs us to the varieties satisfying a non-trivial
congruence identity, equivalently; having a Hobby-McKenzie term. The
corresponding finite property is ``ommiting types (1), (5)''.

\begin{theorem}
Let $\var V$ be a variety satisfying a non-trivial congruence identity
(see~\cite{CongruenceLattices}). Then $\var V$ has a Siggers term.
\end{theorem}
\begin{proof}
Among many characterization (see~\cite{CongruenceLattices}, Theorem A.2), we
use the following one (number (9) in the reffered theorem).
A variety $\var V$ has a Hobby-KcKenzie term if and only if it has
a sequence of terms $f_0, \ldots, f_{2m+1}$ such that
\begin{enumerate}[(i)]
\item $\var V\models f_0(x,y,u,v)\approx x$,
\item $\var V\models f_i(x,y,y,y)\approx f_{i+1}(x,y,y,y)$ for even $i$,
\item $\var V\models f_i(x,x,y,y)\approx f_{i+1}(x,x,y,y)$ and\\
      $\var V\models f_i(x,y,x,y)\approx f_{i+1}(x,y,x,y)$ for odd $i$,
\item $\var V\models f_{2m+1}(x,y,u,v)\approx v$.
\end{enumerate}
We can easily check that all the terms $f_i$ have to be idempotent,
that is $f_i(x,x,x,x) = x$ for any $i$ and $x$. We can assume that
all terms of $\var V$ are idempotent, otherwise we consider the
idempotent reduct.

Let $\alg F$ be a $\var V$-free algebra
generated by an infinite countable set $X$.
By Proposition \ref{loop-prod-char} (ii), it suffices to prove the
following. If $\relstr G=(A,G)$ is a graph, where the binary relation
$G$ is generated by distinct pairs of $X$, then $R$ contains a loop.
Since our
clique is infinite, and every element $x\in\alg F$ is generated by
finitely many elements of $X$, any such element is $\relstr G$-adjacent
to all but finitely many elements of $X$. Therefore, for any finite
set $x_1,x_2,\ldots,x_n\in\alg F$, we can find an element $y\in X$
such that $y\edgev  x_i$ for all $i=1,\ldots,n$.

Let $\hat\phi_1, \hat\phi_2\colon X\to X$ be two injective mappings with
disjoint images $X_1, X_2$, and let
$\phi_1, \phi_2\colon \alg F\to\alg F$ be their unique extensions to
endomorphisms of $\alg F$.

Since $\alg F$ is idempotent and the relation $G$ is full on
$X_1\times X_2$, we have an edge $\phi_1(x)\edgev  \phi_2(y)$ for
every pair $x,y\in\alg F$.
Since $\hat\phi_1, \hat\phi_2$ are injective, the mappings
$\hat\phi_1, \hat\phi_2$ are endomorphisms of the
relational structure $\relstr G|_{X}$. Therefore, even the induced
mappings $\phi_1,\phi_2$ are homomorphisms of the relational structure
$\relstr G$.

For $i=0,1,\ldots,m$,
we gradually construct sequences $\hat a_i, a_i, \hat b_i, b_i, \hat c_i,
c_i, d_i$ such that
$$
a_i \edgev  f_{2i}(x,a_i,a_i,a_i)
$$
for all $x\edgev  \hat a_i$.
We start with an arbitrary $\hat a_0 = a_0$, then $f_0(x,a_0,a_0,a_0) = x$
so the condition on $a_0$ is satisfied.
Now consider a general
$i\in\{0,1,\ldots,m\}$ and take a $\hat b_i\edgev  \hat a_i$. Then
$$
a_i \edgev 
f_{2i}(\hat b_i,a_i,a_i,a_i)=f_{2i+1}(\hat b_i,a_i,a_i,a_i)\defeq b_i.
$$
so $b_i\edgev  a_i$. We continue by taking an element $\hat c_i\in X$ such
that $\hat c_i\edgev  \hat b_i, a_i$. Then
$$
c_i\defeq f_{2(i+1)}(\hat c_i,b_i,\hat c_i,b_i)
= f_{2i+1}(\hat c_i,b_i,\hat c_i,b_i)
\edgev  f_{2i+1}(\hat b_i,a_i,a_i,a_i) = b_i
$$
and
$$
d_i\defeq f_{2(i+1)}(\hat c_i,\hat c_i,b_i,b_i)
= f_{2i+1}(\hat c_i,\hat c_i,b_i,b_i)
\edgev  f_{2i+1}(\hat b_i,a_i,a_i,a_i) = b_i
$$
Finally, consider an element $\hat a_{i+1}\in X$ such that
$\hat a_{i+1}\edgev  \phi_1(\hat c_i),\phi_2(\hat c_i)$, define
$$
a_{i+1} = f_{2(i+1)}(\hat a_{i+1}, \phi_1(c_i), \phi_2(d_i), \phi_2(d_i)),
$$
and check the edges $a_{i+1} \edgev  \phi_1(c_i), \phi_2(d_i)$:
\begin{align*}
f_{2(i+1)}(\hat a_{i+1}, \phi_1(c_i), \phi_2(d_i), \phi_2(d_i))
&\edgev  f_{2(i+1)}(\phi_1(\hat c_i), \phi_1(b_i), \phi_1(\hat c_i), \phi_1(b_i)),\cr
f_{2(i+1)}(\hat a_{i+1}, \phi_1(c_i), \phi_2(d_i), \phi_2(d_i))
&\edgev  f_{2(i+1)}(\phi_2(\hat c_i), \phi_2(\hat c_i), \phi_2(b_i), \phi_2(b_i)),\cr
\end{align*}
So $a_{i+1}\edgev  \phi_1(c_i)$ and
$a_{i+1}\edgev  \phi_2(d_i)$. Therefore for any $x\in X$ such that
$x\edgev  \hat a_{i+1}$, we have an edge
$$
a_{i+1} = f_{2(i+1)}(\hat a_{i+1}, \phi_1(c_i), \phi_2(d_i), \phi_2(d_i))
\edgev  f_{2(i+1)}(x, a_{i+1}, a_{i+1}, a_{i+1}),
$$
so we ensured the condition for $a_{i+1}$.
We ultimately get an edge
$$
a_m \edgev  f_{2m}(x, a_m, a_m, a_m) = f_{2m+1}(x, a_m, a_m, a_m) = a_m
$$
for some $x\edgev  \hat a_m$. This gives a loop on $a_m$ and finishes
the proof.
\end{proof}

\section{Cyclic terms}
\label{cyclic-terms}
        
For $n\geq 2$, an $n$-ary cyclic term $c_n$ is such a term that
satisfies
$$
c(x_1,x_2,x_3\ldots,x_n) = c(x_2,x_3,\ldots,x_n,x_1).
$$
The existence of an $n$-ary cyclic term is clearly a loop condition of
the digraph $\DC_n$. The importance of cyclic terms among loop
conditions lies in the following theorem.

\begin{theorem}
  \label{strong-con-char}
  Every non-trivial strongly conencted loop condition $C$ is either
  equivalent to the existence of a Siggers term,
  or to the existence of a certain cyclic term.
\end{theorem}

Before we prove the theorem, we analyze what implications hold
between the loop conditions of cyclic terms. We can directly get the
following.
\begin{proposition}
  For $d,n\geq2$, where $d$ is a divisor of $n$, the following
  implications between loop conditions hold.
  \begin{enumerate}[(i)]
  \item $\DC_n \Rightarrow \DC_d$,
  \item $\DC_n \Rightarrow \DC_{n^2}$.
  \end{enumerate}
\end{proposition}
\begin{proof}
  The implication (i) follows from the existence of a digraph
  homomorphism $\DC_n\to\DC_d$. For proving the implication (ii), we
  use a standard trick: Let $\relstr G = (A, G)$ be a digraph
  containing $\DC_{n^2}$ and compatible with an $n$-ary cyclic term $c_n$.
  We define a digraph $\relstr H$, where $x\redge H y$ if
  and only if there is a directed $\relstr G$-walk from $x$ to $y$ of
  length $n$. The digraph $\relstr H$ is pp-defined from $\relstr G$,
  so it is compatible with $c_n$ as well. Since $\relstr G$ contained
  $\DC_{n^2}$, the digraph $\relstr H$ contains $\DC_n$, and by
  compatibility with $c_n$, it has a loop. A loop in $\relstr H$
  corresponds to $\DC_n$ in $\relstr G$, so by second application of
  $c_n$, we get a loop in $\relstr G$. This finishes the proof.
\end{proof}

Just from these two facts, we get the following.
Let $\rad(n)$ denote the radical of an integer $n\geq2$, that is the
product of all distinct prime divisors of $n$.

\begin{corollary}
  \label{radical-impl}
  If $\rad(n_2)$ divides $\rad(n_1)$, then the $\DC_{n_1}$ loop
  condition implies the $\DC_{n_2}$ loop condition. In particular, all the
  loop conditions $\DC_n$ with a fixed radical of $n$ are equivalent.
\end{corollary}

Now, we are ready prove Theorem~\ref{strong-con-char}. Consider a
strongly connected loop condition $C$. Let $n = al(\relstr G_C)$
If $n=1$, then $C$ is
equivalent to the existence of a Siggers term
by Corollary~\ref{loop-cond-eq}. Otherwise, there is a digraph
homomorphism $\relstr G_C\to \DC_n$, and a digraph homomorphisms
$\DC_{kn}\to\relstr G_C$ for any $k$ that is large enough. So there is
a digraph homomorphism $\DC_{n^e}\to\relstr G_C$ for some exponent $e$,
and since the $\DC_n$ and the $\DC_{n^e}$ loop conditions are
equivalent, they are both equivalent to $C$.

We finish this section with an example showing that the implications
we have, are optimal.
\begin{theorem}
  Let $C_1, C_2$ be two non-trivial strongly connected loop
  conditions. Let
  numbers $n_1, n_2$ denote the algebraic lengths of
  $\relstr G_{C_1}$, $\relstr G_{C_2}$ respectively. Then $C_1$
  implies $C_2$ if and only if $\rad(n_2)$ divides $\rad(n_1)$.
\end{theorem}
\begin{proof}
  If $\rad(n_2)$ divides $\rad(n_1)$, then either $n_2=1$, or
  $n_1,n_2\geq 2$. In the first case, $C_2$ is the weakest non-trivial
  loop condition, and therefore implied by $C_1$. In the second case,
  the conditions are equivalent to the existence of cyclic terms of
  arities $n_1, n_2$ respectively, and we get $C_1\Rightarrow C_2$
  from Corollary~\ref{radical-impl}.

  Now suppose that $\rad(n_2)$ does not divide $\rad(n_1)$.
  So there is a prime number $p$ that divides $n_2$ but not $n_1$.
  Let $A$ be the universe of the one-dimensional vector space over the
  $p$-element field. We equip the set $A$ with operations of the
  following form:
  For every $k$-tuple $\alpha_1,\ldots,\alpha_k$ such that
  $\alpha_1+\cdots+\alpha_k=1$, we consider the $k$-ary operation
  $A^k\to A$:
  $$
  (x_1,\ldots,x_k)\mapsto \alpha_1x_1+\cdots+\alpha_kx_k,
  $$
  that is, the affine combination given by coefficients
  $\alpha_1\ldots,\alpha_k$. We denote the final algebra
  $(A,\text{affine combinations})$ as $\alg A$. Note that all the term
  operations of $\alg A$ are still just some affine combinations.
  The algebra $\alg A$ satisfies $C_1$. If $n_1\geq 2$, it suffices to
  find a cyclic affine combination. There is the following one:
  $$
  c_1(x_1,\ldots,x_{n_1}) = \frac 1{n_1}\sum_{i=1}^{n_1}x_i.
  $$
  Note that $\frac 1{n_1}$ is well defined since $p$ does not divide $n_1$.
  If $n_1=1$, it suffices to find any satisfied loop condition.
  There is the Maltsev term $m(x,y,z)=x-y+z$ satisfying
  $m(y,x,x)=m(z,z,y)$.

  On the other hand, there is no $n_2$-ary cyclic term in $\alg A$.
  For a contradiction, consider such a term $c_2$:
  $$
  c_2(x_1,\ldots,x_{n_2}) = \sum_{i=1}^{n_2}a_ix_i.
  $$
  Let $0$ denote the zero vector in $\alg A$, and $1$ denote an
  arbitrary non-zero vector. Since $c_2$ is cyclic, all the values
  $$
  c_2(1,0,\ldots,0) = c_2(0,1,0,\ldots,0) = \cdots =
  c_2(0,\ldots,0,1)
  $$
  are equal. Therefore $a_1 = a_2 = \cdots = a_{n_2}$. But then
  $$
  \sum_{i=1}^{n_2}a_i = n_2a_1 = 0
  $$
  since $p$ divides $n_2$. This contradicts the fact that
  $c_2$ should be affine, and the sum is supposed to be equal to 1.
\end{proof}

\section{Without strong conectedness}
\label{no-strong-conn}

The case of loop conditions that are not strongly connected is largely
open. It seems that there is no further big collapse of loop
conditions up to equivalence. We provide two basic counterexamples
concerning loop conditions that are not strongly connected, and then
we provide a few of positive results.

\subsection{Counter-examples}

We have shown that all the non-trivial loop conditions corresponding to a
digraph with a strongly connected component of algebraic length one
are equivalent. In our first example we demonstrate that there
is no other loop condition equivalent to them.

\begin{example}
  There is an algebra $\alg A = (A, s)$ and a compatible digraph
  $\relstr G=(A, G)$ such that $s$ is an idempotent 6-ary Siggers
  operation, $\relstr G$ does not have a loop but any countable digraph
  without a strongly connected component with algebraic length 1 can
  be homomorphicly mapped into $\relstr G$. Therefore $\alg A$
  satisfies all the non-trivial loop conditions corresponding to a
  digraph with a strongly connected component with algebraic length
  one, but it satisfies no other non-trivial loop condition.
\end{example}
\begin{proof}
Let $\relstr G_0 = (A_0, G_0)$ be the disjoint union of directed cycles
of all lengths $l\geq 2$. First, we show that $\relstr G_0$ is
compatible with an idempotent algebra $\alg A_0 = (A_0, s_0)$, where
$s_0$ is a 6-ary Siggers operation.

We define $\sim$ as the smallest reflexive
symmetric binary relation on $A_0^6$ that satisfies
$$
(x,x,y,y,z,z)\sim(y,z,z,x,x,y)
$$
for any $x,y,z\in A_0$. Note that this relation is also
transitive: The only non-trivial way of applying transitivity is by
interpretting a single six-tuple as both $(x,x,y,y,z,z)$ and
$(y,z,z,x,x,y)$, but then $x=y=z$, and we generate just the
reflexivity on constant tuples.

Let $\phi\colon A_0\to A_0$ be the mapping that maps $x$ to $y$ such
that $x\redge {G_0} y$, and let $S$ be a factor set $A_0^6/\sim$. The
function $\phi$ is clearly compatible with $\sim$, so $\phi$ acts on
$S$ as well. In order to construct an idempotent Siggers operation, we
need to find a mapping $s_0\colon S\to A_0$ that maps constant
six-tuples to the appropriate elements and commutes with $\phi$. The
constant tuples has just one-element equivalent classes in $\sim$, and
they are closed under $\phi$, so we can handle them independently of
the rest, and simply map them to the appropriate elements.

Now, consider any other element $\alg x\in S$ represented by a six-tuple
$$
\bar x = (x_1,x_2,x_3,x_4,x_5,x_6)\in A_0^6.
$$
The element $\alg x$ is
in a $\phi$-cycle of a finite length $l$ that is less than or equal to
the product of the lengths of the cycles containing $x_1,\ldots,x_6$. If
$l\geq2$, we can map that cycle into $\relstr G$. It remains to check
that the possibility $l=1$ cannot happen.
For a contradiction, assume that $\phi(\mathbf{x}) = \mathbf{x}$, that
is, $\phi(\bar x) \sim \bar x$. Since $\phi(x_1)\neq x_1$, also
$\phi(\bar x)\neq\bar x$. Therefore there are $x,y,z\in A_0$ such that
$$
\bar x = (x,x,y,y,z,z)\text{ and }\psi(\bar x) = (y,z,z,x,x,y),
$$
where $\psi$ stands for either $\phi$ of $\phi^{-1}$.
Then we have $y=\psi(x)=z$ and $z=\psi(y)=x$, and we get a
contradiction with the assumption that $\bar x$ is non-constant.

Now, we construct the promised example. Take the algebra
$\alg A_1 = (\Q, s)$, where $\Q$ denotes the set of all rational
numbers, and $s_1$ is the arithmetical mean of six elements. Note that
$s_1$ satisfies the Siggers equation as well. We set
$\alg A = \alg A_0 \times \alg A_1 = (A, s)$. The digraph
$\relstr G=(A,G)$ is defined as follows:
$$
(x_1,q_1)\redge G(x_2,q_2)\buildrel\rm def\over\Leftrightarrow
q_1<q_2\text{ or }(q_1=q_2\text{ and }(x_1,x_2)\in G_0).
$$
Now, consider six edges $((x_i,p_i),(y_i,q_i))\in G$ for
$i=1,2,3,4,5,6$. Let $((x,p),(y,q))$ be the result of applying $s$ to
them. We have to check that the resulting edge is in $G$ as well.
We distinguish two cases. If $p_i < q_i$ for some $i$, then $p<q$, so
$(x,p)\redge G (y,q)$. Otherwise for all $i$, $p_i=q_i$ and
$x_i\redge{G_0} y_i$. In that case $p=q$ and $x\redge{G_0} y$, so
$(x,p)\redge G (y,q)$.

Finally, we observe that any countable digraph $\relstr H$ having no
strongly connected component with algebraic length 1 can be
homomorphically mapped to $\relstr G$. Modulo the strong conectedness,
the digraph $\relstr H$ has no loop, so it can be mapped into
$(\Q, <)$, since $(\mathbb{Q}, <)$ contains all strict partial orders
as subdigrahs. Finally, every strongly connected component of
$\relstr H$ can be mapped into $G_0$ since it has algebraic length
bigger than one. These two mappings together form the digraph
homomorphism $\relstr H\to\relstr G$.
\end{proof}

Our second example shows that if we do not require at least
some cycles in the digraph, we cannot get a loop even under quite strong
assumptions. In particular, our example is locally finite, has a
ternary near unanimity operation, and the compatible digraph forms an
$\omega$-categorical structure.

\begin{example}
  Let $\alg A=(\{0,1\}, m)$ be the two-element algebra, where $m$ is
  the majority operation
  $$
  m(x,x,y)=m(x,y,x)=m(y,x,x)=x.
  $$
  There is a smooth loopless digraph $\alg G$ compatible with a subalgebra
  $\alg B\leq \alg A^\omega$, such that $\relstr G$ contains as subdigraphs all
  countable digraphs without directed cycles.
\end{example}
\begin{proof}
  Regard elements of $\alg A^\omega$ as
  infinite sequences indexed by rational numbers, that is functions
  $a\colon\Q\to\{0,1\}$. For $q\in\Q$ define $b_q\in\alg A^\omega$ as
  follows
  $$
  b_q(x) =
  \begin{cases}
    0 & \text{if }x\leq q,\cr
    1 & \text{if }x > q.\cr
  \end{cases}
  $$
  Then if $q_1,q_2,q_3\in\Q$, the value $m(b_{q_1},b_{q_2},b_{q_3})$
  equals the function $b_q$, where $q$ is the median element of the
  triple $(q_1,q_2,q_3)$. Therefore, the set $B=\{b_q:q\in\Q\}$ is a
  subuniverse, and we consider the subalgebra $\alg B=(B,m)$.

  The digraph $\relstr G=(B,G)$ is constructed by putting
  $(b_{q_1},  b_{q_2})\in G$ if and only if $q_1 < q_2$. This digraph
  contains all countable digraphs without directed cycles since any
  countable strict partial order is a subdigraph of $(\Q, <)$.

  It is not difficult to check that $\relstr G$ is compatible with the
  operation $m$, that is that the strict order is compatible with
  median operation. Let $m'$ denote the median of three rational
  numbers. Consider three pairs $p_1<q_1$, $p_2<q_2$, $p_3<q_3$, we
  check $m'(p_1,p_2,p_3) < m'(q_1,q_2,q_3)$. Without loss of generality,
  $p_1\leq p_2\leq p_3$, so $m'(p_1,p_2,p_3) = p_2$. Both $q_2$
  and $q_3$ are strictly greater than $p_2$, therefore
  $m'(q_1,q_2,q_3) > p_2 = m'(p_1, p_2, p_3)$.
\end{proof}

\subsection{Partial positive results}

Here we show a few loop conditions that are not strongly connected,
and are satisfied by certain algebras. We begin with the definitions.

A \emph{Maltsev term} is an idempotent ternary term $m$ satisfying the
equation $m(y,x,x) = m(z,z,y)$. By a \emph{non-idempotent Maltsev
term} we mean a term $m$ satisfying this equation but not being
necessarily idempotent. The existence of a Maltsev term is a loop
condition given by the graph $x\to y\to z\leftarrow x$.

A \emph{near unanimity term} (briefly \emph{NU term}) is an $n$-ary term $t$ such
that $t(x,\ldots,x,y,x,\ldots,x)=x$ for any position of $y$. Here, we
will be primarily interested in ternary NU term, that is, a term $t$
such that $t(x,x,y)=t(x,y,x)=t(y,x,x)=x$.

By a \emph{cone}, we mean the digraph given by
$c_1\redgev\cdots\redgev c_n\to c_1$, where $n\geq 2$,
$a\redgev c_i$ for any $1\leq i\leq n$, and $b\redgev a$.

\begin{proposition}
  Non-idempotent Maltsev term implies any weakly connected loop
  condition with algebraic length 1.
\end{proposition}
\begin{proof}
  Consider a digraph $\relstr G=(A,G)$ having a weakly connected
  component of algebraic length 1,
  and compatible with an algebra $\alg A$ with a Maltsev term
  $m$. Since
  $\relstr G$ has algebraic length 1, there is an oriented cycle
  $a_0,a_1,\ldots,a_{n-1},a_0$ such that $a_i\redge G a_{i+1}$ or
  $a_i\ledge G a_{i+1}$ for every $i=0.\ldots,n-1$ (we calculate
  indices modulo $n$ in this proof), and the number or right
  arrows exceeds the number of left arrows by one.
  We prove our proposition by induction on
  $n$. If $n=1$, we have a loop.

  If $n>1$, we distinguish two cases.
  \begin{enumerate}[(a)]
  \item There is a zig-zag pattern in the cycle
    $a_i\redge G a_{i+1}\ledge G a_{i+2}\redge G a_{i+3}$, or
    $a_i\ledge G a_{i+1}\redge G a_{i+2}\ledge G a_{i+3}$.
  \item There is no zig-zag pattern in the cycle, that is, each arrow
    is next to an arrow of the same direction.
  \end{enumerate}
  If there is a zig-zag pattern, we pp-define a digraph $\relstr H=(A,H)$
  by $y_1\redge H y_2$ if there are $x,z$ such that
  $y_1\redge G z\ledge G x\redge G y_2$. Notice that $H\supset G$, and
  $\relstr H$ builds a ``bridge'' $a_i\redge H a_{i+3}$, or
  $a_{i+3}\redge H a_i$ to the zig-zag pattern. So if we omit
  $a_{i+1},a_{i+2}$, the cycle has
  still algebraic length one in $\relstr H$. By induction hypothesis,
  Maltsev term imply the loop condition of that shorter cycle, so
  $\relstr H$ has a loop. A loop in $\relstr H$ corresponds to
  $x\redge G y\redge G z, x\redge G z$, which is the digraph of the
  non-idempotent maltsev term. Therefore $\relstr G$ has a loop.

  If there is no zig-zag pattern in the cycle, we use the following
  pp-definition. $x_1\redge H x_2$ if there are $y,z$ such that
  $x_1\redge G y\redge G z\ledge G x_2$. Now, we have $x_1\redge H x_2$ whenever
  $x_1\redge G x_2$ and $x_2$ has an outcoming
  edge. Nevertheless, we can omit all the peaks from our
  cycle. Whenever we have
  $$a_i\redge G a_{i+1}\redge G a_{i+2}\ledge G a_{i+3}\ledge G a_{i+4},$$
  we can discard $a_{i+2}$ since $a_i\redge H a_{i+3}$. Therefore we
  get a shorter cycle in $\relstr H$ following a loop in $\relstr H$.
  As in the previous case, a loop in $\relstr H$ corresponds to a
  digraph $x\redge G y\redge G z, x\redge G z$ that implies a loop using
  the Maltsev term.
\end{proof}

\begin{proposition}
  Ternary near unanimity term implies the ``cone'' loop condition.
\end{proposition}
\begin{proof}
  Let $\alg A=(A, t)$ be an algebra where $t$ is a ternary near
  unanimity term, and assume that there is a graph (denoted by
  arrows) compatible with $\alg A$. 
  Moreover assume that a cone $a,b,c_1,\ldots,c_n$ is a subgraph of
  the graph. We need to prove that the graph contains a loop.

  We set $x_n = c_n$ and
  recursively construct elements
  $x_{n-1}, x_{n-2}, \ldots, x_1\in\alg A$:
  by $x_i = m(x_{i+1}, c_i. a)$.

  We get a loop on $x_1$ by proving the following properties by induction
  \begin{enumerate}[(i)]
  \item $x_i\redgev c_1$ for all $n\geq i\geq 1$,
  \item $a \redgev x_i$ for all $n\geq i\geq 1$,
  \item $x_i\redgev x_{i+1}$  for all $n-1\geq i\geq 1$.
  \end{enumerate}

  When they are proved, there is a loop by
  $$
  x_1 = m(x_1,x_1,b) \redgev m(x_2,c_1,a) = x_1.
  $$

  The proof is straighforward. We check the first step of induction:
  \begin{enumerate}[(i)]
  \item $x_n = c_n\redgev c_1$,
  \item $a \redgev c_n = x_n$,
  \item $x_{n-1} = m(x_n,c_{n-1},a)\redgev m(c_1,c_n,c_n) = x_n$,
  \end{enumerate}
  and the induction step:
  \begin{enumerate}[(i)]
  \item $x_{i-1} = m(x_i,c_{i-1},a) \redgev m(c_1,c_i,c_1) = c_1$,
  \item $a = m(a,a,b) \redgev m(x_i,c_{i-1},a) = x_{i-1}$,
  \item
    \begin{align*}
      x_{i-1} = m(x_i,c_{i-1},a) &\redgev m(x_{i+1},c_i,x_{i+1}) = x_{i+1},\cr
      x_{i-1} = m(x_i,c_{i-1},a) &\redgev m(x_{i+1},c_i,c_i) = c_i,\cr
      x_{i-1} = m(x_{i-1},x_{i-1},b) &\redgev m(x_{i+1},c_i,a) = x_i.\cr
    \end{align*}
  \end{enumerate}
  This finishes the proof.
\end{proof}

However, we don't have an answer to the following question.
\begin{question}
  Does the existence of a 4-ary near unanimity term
  $$
  t(y,x,x,x)=t(x,y,x,x)=t(x,x,y,x)=t(x,x,x,y)=x
  $$
  imply any loop condition that is not equivalent to the $\relstr K_3$
  loop condition?
\end{question}

\bibliographystyle{plain}
\bibliography{bib-file.bib}

\end{document}